\documentclass[12pt]{amsart}
\usepackage{amsmath,amsfonts,amssymb,graphics,mathrsfs,amsthm,eurosym,verbatim,enumerate,quotes,graphicx}
\usepackage[marginratio=1:1,tmargin=117pt, height=650pt]{geometry}
\usepackage[usenames, dvipsnames]{color}

\usepackage[flushmargin]{footmisc}
\usepackage[noadjust]{cite}
\numberwithin{equation}{section}
\makeatletter
\def\blfootnote{\xdef\@thefnmark{}\@footnotetext}
\makeatother
\theoremstyle{plain}
\newtheorem{theorem}{Theorem}[section]
\newtheorem{proposition}[theorem]{Proposition}
\newtheorem{cor}[theorem]{Corollary}

\newtheorem{lemma}[theorem]{Lemma}
\newtheorem{definition}[theorem]{Definition}

\newtheorem{claim}[theorem]{Claim}

\newtheorem*{remark}{Remark}
\newtheorem*{remarks}{Remarks}

\newcommand{\C}{{\mathbb{C}}}

\newcommand{\B}{\mathcal B}

\newcommand{\N}{{\mathbb{N}}}

\newcommand{\s}{\underline{s}}
\newcommand{\tauinf}{\overset{\infty}{\tau}}

\begin{document}
\title[Singularities of inner functions]{Singularities of inner functions associated with hyperbolic maps}
\author{Vasiliki Evdoridou, \, N\'uria Fagella, \, Xavier Jarque, \, David J. Sixsmith}
\address{School of Mathematics and Statistics\\ The Open University \\ Walton Hall \\ Milton Keynes MK7 6AA \\ UK}
\email{vasiliki.evdoridou@open.ac.uk}
\address{Departament de Matem\'atiques i Inform\'atica \\ Institut de Matem\'atiques de la Universitat de Barcelona (IMUB) and Barcelona Graduate School of Mathematics (BGSMath) \\  Gran Via 585 \\ 08007 Barcelona \\ Catalonia}
\email{nfagella@ub.edu}
\address{Departament de Matem\'atiques i Inform\'atica \\ Institut de Matem\'atiques de la Universitat de Barcelona (IMUB) and Barcelona Graduate School of Mathematics (BGSMath) \\  Gran Via 585 \\ 08007 Barcelona \\ Catalonia}
\email{xavier.jarque@ub.edu}
\address{Dept. of Mathematical Sciences \\ University of Liverpool \\ Liverpool L69 7ZL \\ UK \\ ORCiD: 0000-0002-3543-6969} 
\email{djs@liverpool.ac.uk}

\begin{abstract}
%
Let $f$ be a function in the Eremenko-Lyubich class $\B$, and let $U$ be an unbounded, forward invariant Fatou component of $f$. We relate the number of singularities of an inner function associated to $f|_U$ with the number of tracts of $f$. In particular, we show that if $f$ lies in either of two large classes of functions in $\B$, and also has finitely many tracts, then the number of singularities of an associated inner function is at most equal to the number of tracts of $f$. 

Our results imply that for hyperbolic functions of finite order there is an upper bound -- related to the order -- on the number of singularities of an associated inner function.
\end{abstract}
\maketitle
\section{Introduction}
Let $f$ be a transcendental entire function. We denote by $f^n$ the $n$th iterate of $f$. The set of points for which the iterates $\{f^n\}_{n \in \mathbb{N}}$ form a normal family in some neighbourhood is called the \textit{Fatou set} $F(f)$. Its complement in the complex plane is the \textit{Julia set} $J(f)$. The Fatou set is open and consists of connected components which are called \textit{Fatou components}. For an introduction to these sets and their properties see \cite{bergweiler93}.

Let $U \subset \mathbb{C}$ be an unbounded \emph{forward invariant} Fatou component; in other words, $f(U) \subset U$. Note that it follows from \cite[Theorem 3.1]{Baker} that $U$ is simply connected. Let also $\phi : \mathbb{D} \to U$ be a Riemann map. Then the function $h: \mathbb{D} \to \mathbb{D}$ defined by $h := \phi^{-1} \circ f \circ \phi$ is an \emph{inner function associated to} $f \lvert _{U}$. Note that $h$ is unique up to a conformal conjugacy.

A point $\zeta \in \partial \mathbb{D}$ is a \textit{singularity} of $h$ if $h$ cannot be extended holomorphically to any neighbourhood of $\zeta$ in $\mathbb{C}.$ This definition is independent of the choice of $\phi$, up to a M\"{o}bius map (see Section \ref{section:prel}). In this paper we are interested in the number of singularities of an associated inner function $h$. In particular, we give two conditions on $f$ and $U$ which ensure that the number of singularities of $h$ is \emph{finite}. This is far from being the case for inner functions in general, for which, \emph{a priori}, every point of $\partial \mathbb{D}$ can be a singularity; this follows, for example, from \cite[Theorem II.6.2]{Garnett}. 

Inner functions have been widely studied in relation to the iteration of transcendental entire functions. Devaney and Goldberg \cite{DevandG} considered the Julia set of $f(z) := \lambda e^z$ in the case that $f$ has a completely invariant attracting basin $U$. Their study used an inner function associated to $f|_U$, for which they found a specific formula. Inner functions were also used as a tool by Bara\'nski in \cite{Baranski}, and Bara\'nski and Karpi\'nska in \cite{BK}, when studying the Julia set of disjoint type functions. Several further results about inner functions associated to transcendental entire functions with an invariant Fatou component were obtained in \cite{BakerDominguez}, \cite{univalentbd}, \cite{Bargmann}, and \cite{fagella-henriksen}.

Our results extend this study, and so are interesting in themselves; indeed, we show that in many cases the associated inner function is very well-behaved on the boundary, having only finitely many singularities. Our results are also pertinent to the results of \cite{Accesses}. In that paper the authors obtained several results on the existence and the number of accesses to boundary points from simply connected Fatou components of certain \emph{meromorphic} functions. In order to obtain some of their results (see, for example, \cite[Theorems B and C]{Accesses}) they introduced a hypothesis on the meromorphic function $f$, which they termed being \emph{singularly nice}. (We define this property in Section \ref{section:prel}). This property, which is not easy to check, is closely related to the singularities of the associated inner function $h$, and is satisfied when all the singularities of $h$ are isolated. Hence all entire functions that satisfy the hypotheses of our statements are also singularly nice.

Our results apply in the very well-studied class of transcendental entire functions known as the Eremenko-Lyubich class, denoted by $\B$. To define this class, for a transcendental entire function $f$, we denote by $S(f)$ the set of \emph{singular values} of $f$. This is the closure of the set of critical and finite asymptotic values of $f$. The class $\B$ consists of those transcendental entire functions $f$ such that $S(f)$ is bounded. A survey of dynamics in the class $\mathcal{B}$ was given in \cite{DaveclassB}. 

If $f \in \B$, then we can choose a Jordan domain $D$ containing all the singular values. The components of $f^{-1}(\C\setminus \overline{D})$ are Jordan domains whose boundary passes through infinity, called \emph{tracts}. In our results we relate the number of singularities of $h$ to the number of tracts of $f$. Our first such result is as follows.

\begin{theorem}
\label{theo:finitelymanytracts}
Suppose that $f \in \B$ has finitely many tracts, and that $U$ is a forward invariant Fatou component of $f$. Suppose also that there is a Jordan domain $D$ such that $S(f) \subset D \subset F(f)$. Then the number of singularities of an associated inner function is equal to the number of tracts of $f$.
\end{theorem} 

\begin{remark}\normalfont
With the hypotheses of Theorem~\ref{theo:finitelymanytracts}, $U$ is unbounded, and is either an attracting or a parabolic basin; see Theorem~\ref{theo:comps} below.
\end{remark}

As a corollary of Theorem~\ref{theo:finitelymanytracts} we give a broad class of functions for which it is easy to check that the hypotheses of that theorem are satisfied. First we recall that the order of a transcendental entire function $f$ is defined by
\[
\rho (f) := \limsup_{r \to \infty} \frac{\log \log M(r,f)}{\log r},
\]
where 
\[
M(r,f) := \max_{\lvert z \rvert =r} \lvert f(z) \rvert, \quad\text{for } r>0.
\]

The following theorem is an immediate consequence of the celebrated Denjoy-Carleman-Ahlfors theorem (see, for example, \cite{nevanlinna}), and gives an upper bound to the number of tracts for a finite-order function in $\B$. (Note that if $f \in \mathcal{B}$, then $\rho(f) \geq 1/2$; see \cite[Lemma 3.5]{rs-dimensions}).

\begin{theorem}
\label{D-C-A}
Let $f \in \B$ be of order $\rho(f) < \infty$. Then $f$ has at most $2\rho(f)$ tracts.
\end{theorem}

A function $f \in \B$ is said to be of \emph{disjoint type} if $D$ can be chosen so that $f(\overline{D})~\subset~D$. It is known that for a disjoint-type function, the Fatou set is connected and unbounded \cite[Proposition 2.8]{HMB}. The following is an immediate consequence of Theorem~\ref{theo:finitelymanytracts}, together with this observation and Theorem \ref{D-C-A}.
\begin{cor}
\label{disjoint-type}
Suppose that $f$ is a disjoint-type function of order $\rho(f) < \infty$, and that $U$ is the Fatou component of $f$. Then the number of singularities of an associated inner function is at most equal to $2\rho(f)$.
\end{cor}
\begin{remark}\normalfont
Suppose that $f \in \B$. Then, for sufficiently large values of $M>0$, the function $g(z) := f(z/M)$ is of disjoint type \cite[p.261]{Rigidity}. It follows that any finite order function in the class $\B$ is equivalent to a function which satisfies the conditions of Corollary~\ref{disjoint-type}.
\end{remark}
Our second main result concerns a rather larger and particularly well-studied class of functions in the class $\mathcal{B}$.
\begin{definition}
A transcendental entire function is called \emph{hyperbolic} if the \emph{postsingular set} defined by 
\[
\mathcal{P}(f) := \overline{\bigcup_{j \geq 0} f^j(S(f))},
\]
is a compact subset of the Fatou set.
\end{definition}
\begin{remark}\normalfont
It is known that all hyperbolic functions lie in the class $\B$; see \cite[Theorem and Definition 1.3]{classBornotclassB}. In addition, a function $f$ is of disjoint type if it is hyperbolic and $F(f)$ is connected \cite[Proposition 2.8]{HMB}).
\end{remark}

Our result concerning hyperbolic functions is as follows.

\begin{theorem}
\label{theo:hyperbolic}
Suppose that $f$ is hyperbolic. Suppose also that $U$ is an unbounded forward invariant Fatou component of $f$. Then the number of singularities of an associated inner function is at most equal to the number of tracts of $f$.
\end{theorem}

The next corollary is an immediate consequence of Theorem \ref{theo:hyperbolic}, once again using Theorem \ref{D-C-A}.

\begin{cor}
\label{cor:hyp}
Suppose that $f$ is hyperbolic, and of order $\rho(f) < \infty$, and that $U$ is an unbounded forward invariant Fatou component of $f$. Then the number of singularities of an associated inner function is at most equal to $2\rho(f)$.
\end{cor}

\begin{remarks}\normalfont
\mbox{ }
\begin{enumerate}
\item In general, if $U$ is bounded, then an associated inner function can be shown to be a finite Blaschke product, and therefore to have no singularities on $\partial \mathbb{D}$.
\item An example of a function that satisfies the hypotheses of Theorem \ref{theo:finitelymanytracts} but not those of Theorem~\ref{theo:hyperbolic} is the function $f(z)= e^z-1$. It is easy to show that $F(f)$ consists of a single completely invariant parabolic basin which contains the only finite singular value; i.e. the asymptotic value $-1$. However, $P(f) \cap J(f) = \{0\}$, and so $f$ is not hyperbolic.
\item In fact we could prove our results for a larger class of functions, that includes the hyperbolic functions. These are the so-called \emph{strongly subhyperbolic functions} introduced in \cite{HMB}. For clarity of exposition, we restrict to the class of hyperbolic functions, as this is a more well-known class.
\end{enumerate}
\end{remarks}

We close by showing, among other things, that the hypotheses of Theorem \ref{theo:finitelymanytracts} imply that the Fatou set is connected. In particular, we can then deduce that the forward invariant Fatou component in the statement of that result is unbounded, which is why there appears to be an additional condition on $U$ in the statement of Theorem~\ref{theo:hyperbolic}. We say that a Fatou component $U$ is \emph{completely invariant} if $z \in U$ if and only if $f(z) \in U$. 
\begin{theorem}
\label{theo:comps}
Let $f \in \B$, let $U$ be a forward invariant Fatou component of $f$, and let $D$ be a Jordan domain such that $S(f) \subset D \subset F(f)$. Then $D \subset U = F(f)$, $U$ is completely invariant and so unbounded, and $U$ is either an attracting or a parabolic basin. If $U$ is an attracting basin, then $f$ is disjoint-type.
\end{theorem}
%
%
\subsection*{Structure} In Section \ref{section 2} we discuss various preliminary topics, such as singularities of inner functions and accesses to boundary points. In Section~\ref{sec:fatou} we prove Theorem~\ref{theo:comps}, and then in Section \ref{sec:disjoint} we prove Theorem \ref{theo:finitelymanytracts}. Finally, in Section~\ref{sec:hyp} we prove Theorem~\ref{theo:hyperbolic}.

\subsection*{Notation} We let $\hat{\C} := \C \cup \{\infty\}$. \\

\emph{Acknowledgments:} 
We would like to thank Chris Bishop for comments which motivated this work. We would also like to thank Krzysztof Bara\'{n}ski, Anna Miriam Benini, Walter Bergweiler, Chris Bishop, Lasse Rempe-Gillen, Phil Rippon and Gwyneth Stallard for helpful discussions. The first author would like to thank the LMS and the IMUB for supporting her stay at Universitat de Barcelona, where this work started. We are grateful to Anne Sixsmith for making cake.
%
%
%
%
%
\section{Preliminaries}
\label{section 2}
\subsection{Accesses to infinity}
\label{section:prel}
Let $U$ be an unbounded, simply connected domain. Following \cite{Accesses} we define an access to infinity from $U$. In the case that $U$ is a forward invariant Fatou component, accesses to points on $\partial U$ from $U$ are a particularly interesting subject of study. In this paper we are only interested in accesses to infinity from $U$ because these are the only accesses that are related to the singularities of a related inner function, $h$; see the remarks at the end of Subsection~\ref{subs:inner}. 
\begin{definition}
Infinity is \emph{accessible} from $U$, if there
exists a curve $\gamma : [0, 1] \to \hat{\C}$ such that $\gamma ([0, 1)) \subset U$ and $\gamma (1) = \infty$. We also say that $\gamma$ \emph{lands} at $\infty$.
Fix a point $z_0 \in U$ and suppose $\infty$ is accessible from $U$. A homotopy class (with fixed endpoints)
of curves $\gamma : [0, 1] \to \hat{\C}$ such that $\gamma([0, 1)) \subset U$, $\gamma (0) = z_0, \gamma (1) = \infty$ is called an \emph{access to infinity} from U.
\end{definition}

We will need the following result about curves belonging to the same access to infinity from $U$, which is \cite[Lemma 4.1(a)]{Accesses}. 
\begin{lemma}
\label{lem:accesses-tracts}
 Let $U \subset \C$ be an unbounded simply connected domain and $z_0 \in U$. Let
$\gamma_0, \gamma_1 : [0, 1] \to \hat{\C}$ be curves such that $\gamma_j ([0, 1)) \subset U, \gamma_j (0) = z_0, \gamma_j (1) = \infty$ for $j = 0, 1$. Then $\gamma_0$ and $\gamma_1$ are in the same access to infinity if and only if there is exactly one
component of $\hat{\C} \setminus (\gamma_0 \cup \gamma_1)$ intersecting $\partial U$.
\end{lemma}

A key ingredient in our arguments is the following well-known result (see \cite[Section 2]{Accesses}), which we state only in the case where the boundary point is infinity.

\begin{theorem}
\label{correspondence theorem}
Let $U \subset \C$ be an unbounded simply connected domain.
Then there is a one-to-one correspondence between accesses from $U$ to infinity and points $\zeta \in \partial \mathbb{D}$,
such that the radial limit of a Riemann map $\phi: \mathbb{D} \to U$ at $\zeta$ exists and is equal to infinity. The correspondence is given as
follows.
\begin{itemize}
\item[(a)] If $A$ is an access to $\infty$, then there is a point $\zeta \in \partial \mathbb{D}$, such that the radial limit
of $\phi$ at $\zeta$ is equal to $\infty$, and for every $\gamma \in A$, the curve $\phi ^{-1}(\gamma)$ lands at $\zeta$. Moreover,
different accesses correspond to different points in $\partial \mathbb{D}$.
\item[(b)] If the radial limit of $\phi$ at a point $\zeta \in \partial \mathbb{D}$ is equal to $\infty$, then there exists an
access $A$ to $\infty$, such that for every curve $\eta \subset \mathbb{D}$ landing at $\zeta$ if $\phi (\eta)$ lands at some
point $w \in \hat{\C}$, then $w = \infty$ and $\phi (\eta) \in A$.
\end{itemize}
\end{theorem}

\subsection{Inner functions}
A holomorphic self-map of the unit disc $h$ is called an \textit{inner function} if radial limits exist at almost all points of the unit circle, and belong to the unit circle. Although radial limits exist for almost all points in $\partial \mathbb{D}$ the behaviour of $h$ on $\partial \mathbb{D}$ can be very irregular. Recall that a point $\zeta \in \partial \mathbb{D}$ is a \textit{singularity} of $h$ if $h$ cannot be extended holomorphically to any neighbourhood of $\zeta$ in $\mathbb{C}.$ We denote the set of singularities of $h$ by $\operatorname{Sing}(h)$. 

We will make use of the following well-known characterization of the set $\operatorname{Sing}(h)$; this follows from, for example, \cite[Theorem II.6.1 and Theorem II.6.4]{Garnett}.
\begin{lemma}
\label{lem:sing-preimages}
Let $h: \mathbb{D} \to \mathbb{D}$ be an inner function. Then, for almost all $z \in \mathbb{D}$, $\operatorname{Sing}(h)$ coincides with the set of accumulation points of the set $h^{-1}(z)$ 
\end{lemma}

\subsection{Inner functions associated to entire functions}
\label{subs:inner}
Let $U$ be an unbounded, forward invariant Fatou component of a transcendental entire function $f$, and let $h$ be an inner function associated to $f|_U$, that is $h := \phi^{-1} \circ f \circ \phi$, where $\phi : \mathbb{D} \to U$ is a Riemann map. Note that the definition of $\operatorname{Sing}(h)$ given above is independent of the choice of $\phi$ up to a M\"{o}bius map, which we show as follows. If $\tilde{\phi}: \mathbb{D} \to U$ is another Riemann map then $\tilde{\phi}= \phi \circ M$, where $M$ is a conformal automorphism of $\mathbb{D}$, and so is a M\"{o}bius map. Hence $h$ and $\tilde{h} := \tilde{\phi}^{-1} \circ f \circ \tilde{\phi} $ are conjugate by a M\"{o}bius map. Thus $\zeta \in \operatorname{Sing}(h)$ if and only if $M^{-1}(\zeta) \in \operatorname{Sing}(\tilde{h})$. 

From the definition of singularities it is obvious that all points on $\partial \mathbb{D}$ for which the radial limit of $h$ does not exist are singularities of $h$, but they are not the only ones. As mentioned in the introduction, if $U$ is bounded, then $h$ is a finite Blaschke product and hence has no singularities on $\partial \mathbb{D}$. Moreover, if $U$ is unbounded and $\deg f|_U= \infty$, then $g$ has at least one singularity on $\partial \mathbb{D}$ (see \cite[Section 3]{Accesses}).

Following the terminology used in \cite{Accesses}, we say that $f \lvert _U$ is \textit{singularly nice} if there exists a singularity $\zeta \in \partial \mathbb{D}$ of $h$ such that the angular derivative of $h$ is finite at every point $z$ in some punctured neighbourhood of $\zeta$ in $\partial \mathbb{D}$. (For a definition of angular derivative see, for example, \cite[p.42]{Garnett}.) It follows from the definition that $f$ is singularly nice whenever all the singularities of $h$ are isolated. In the cases studied in this paper we show the stronger property that $h$ has finitely many singularities.

It follows from Lemma \ref{lem:sing-preimages} and Theorem \ref{correspondence theorem} that the singularities of $h$ are related only to accesses from  $U$ to infinity. Indeed, we look for boundary points of $U$ where the preimages of almost all points in $U$ accumulate. This can only be infinity because preimages of a point by $f$ cannot accumulate at a finite point.

\subsection{Tracts of functions in the class $\B$}
\label{subsection: class B}
We now give some general background on tracts and fundamental domains for functions in the class $\B$; these definitions are now standard. Let $D$ be a Jordan domain containing all the singular values of $f$. Each component of $f^{-1}(\mathbb{C}\setminus\overline{D})$ is known as a \emph{tract}. It is well-known that each tract is a simply connected Jordan domain whose boundary is a Jordan arc tending to infinity in both directions. Note that in all the cases we consider, $f$ has finitely many tracts; say $n$. We denote these tracts by $T_i$, where $i=1,\dots, n$. The restriction of $f$ to each tract is a universal covering onto $\mathbb{C}\setminus\overline{D}$. We let $\mathcal{T} = \bigcup_{i=1}^n T_i$ denote the union of the tracts.

Let $\delta$ be a simple curve in $\mathbb{C}\setminus\overline{D}$ which connects $\partial D$ to infinity and does not intersect $\overline{\mathcal{T}}$; it is easy to see that such a curve exists. Then $f^{-1}(\delta)$ cuts every tract into countably many components, which we call \textit{fundamental domains}. Note that if $T$ is a tract and $F \subset T$ is a fundamental domain, then $f$ maps $F$ conformally to $\mathbb{C}\setminus (\overline{D} \cup \delta)$, and maps $\partial F \cap \partial T$ onto $\partial \mathbb{D}.$

\begin{figure}
	\includegraphics[width=16cm,height=10cm]{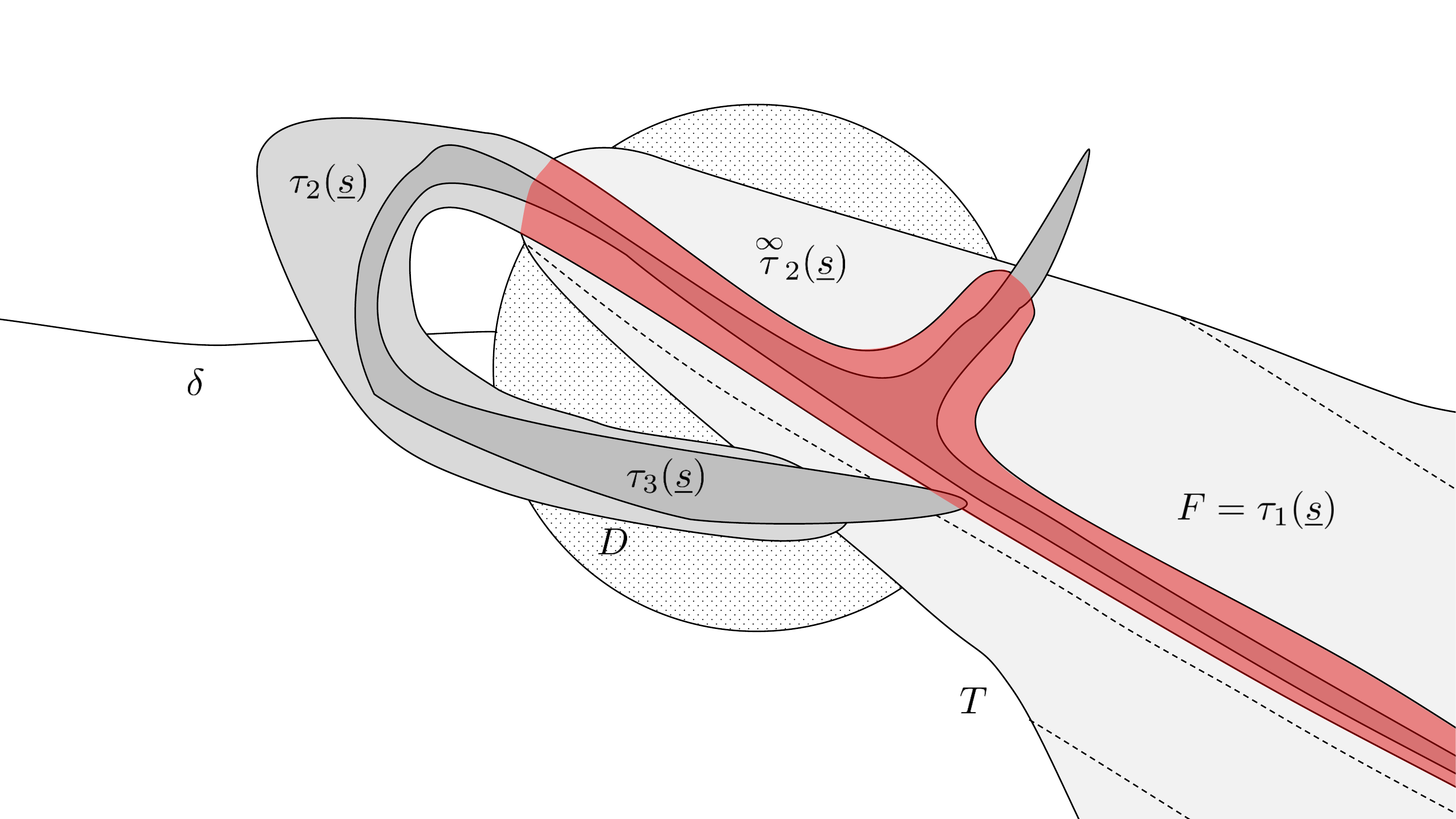}
  \caption{This illustration shows a tract, $T$, a fundamental domain, $F$, which is also a fundamental tail $\tau_1(\s)$, and two other fundamental tails, $\tau_2(\s)$ and $\tau_3(\s)$. The set $\tauinf_2(\s)$ is the unbounded component of $\tau_2(\s) \cap \tau_1(\s)$, shown in red.}\label{fig:tractsgraphic}
\end{figure}

\subsection{Fundamental tails}
\label{subs:tails}
The concept of fundamental tails was introduced in \cite{AnnaLasse}, and we need some background and definitions from that paper; see also Figure~\ref{fig:tractsgraphic}. Note that although these apply in general whenever $P(f)$ is bounded, in this paper we will only consider the case that $f$ is hyperbolic. 

When proving Theorem~\ref{theo:hyperbolic} we will let $D$ be a Jordan domain containing $P(f)$; in fact we will make $D$ even larger than this. We continue to let $\delta$ be a simple curve in $\C \setminus \overline{\mathcal{T}}$ that joins a point of $\partial D$ to infinity, as used to define the fundamental domains. Note that $\overline{F} \cap P(f)= \emptyset$, for every fundamental domain $F$, since $f(\overline{F}) \cap D = \emptyset$ and $P(f)$ is forward invariant. 

Let $W_0 := \C \setminus (\overline{D} \cup \delta)$. For each $n \in \N$, a connected component, $\tau$, of $f^{-n}(W_0)$ is called a \emph{fundamental tail of level $n$}. As for tracts, each fundamental tail is a simply connected Jordan domain whose boundary is a Jordan arc tending to infinity in both directions. Since $W_0$ is simply connected, and does not meet $P(f)$, $f^n$ is a conformal map from $\tau$ to $W_0$. Note that the fundamental tails of level $1$ are, in fact, the fundamental domains of $f$. It is straightforward to show that the image of a fundamental tail of level $n > 1$ is a fundamental tail of level $n-1$.

An \emph{external address} is an infinite sequence of fundamental domains $\s= F_0 F_1 \ldots$. For an external address $\s$, we denote by $\tau_{n}(\s)$ the unique fundamental tail of level $n$ such that $f^k(\tau_n(\s))$ has unbounded intersection with $F_k$, for $0 \leq k \leq n-1$. In this case, we say that $\tau_n(\s)$ \emph{has external address} $\s$. The existence and uniqueness of $\tau_n(\s)$ are demonstrated in \cite{AnnaLasse}. Observe, for example, that $\tau_1(\s) = F_0$, and that $\tau_2(\s)$ is the preimage of $F_1$ that has unbounded intersection with $\tau_1(\s)$.

Suppose that $\s$ is an external address, and that $n \in \N$. It can be shown that $\tau_{n+1}(\s)\cap \tau_n(\s)$ has a unique unbounded component. (It is possible for this intersection also to have bounded components; see Figure~\ref{fig:tractsgraphic}.) We let $\tauinf_{n+1}(\s)$ denote the unbounded connected component of $\tau_{n+1}(\s)\cap \tau_n(\s)$.

We are particularly interested in points whose orbit \emph{eventually} stays in these sets $\tauinf_n(\s)$. Accordingly, we say that a point $z \in \C$ \emph{has external address} $\s$ if $z \in \tauinf_{n}(\s)$ for all sufficiently large $n$. We denote the set of all points $z \in \C$ having external address $\s$ by $J_{\s}(f)$. It is easily seen that if $z \in J_{\s}$, then $f^n(z)$ lies in the (unique) unbounded component of $F_n \setminus D$ for all sufficiently large $n$. It can then be shown that $J_{\s} \subset J(f)$. In general, however, there are points in the Julia set that do not have an external address.

\subsection{Topological result}
We require the following topological result \cite[Theorem 5.6]{Nadler}, which is known as a ``Boundary Bumping Theorem''. 
\begin{theorem}
\label{nad}
Suppose that $X'$ is a continuum, that $E$ is a non-empty proper subset of $X'$, and that $K$ is a component of $E$. Then $\overline{K} \cap \overline{X' \setminus E} \ne \emptyset$.
\end{theorem}
%
%
%
\section{Dynamics of functions satisfying the conditions of Theorem~\ref{theo:finitelymanytracts}}
\label{sec:fatou}
In this section we characterise the possible components of the Fatou set of a function that satisfies the hypotheses of Theorem~\ref{theo:finitelymanytracts}. We first prove the following simple proposition.
\begin{proposition}
\label{prop:comps}
Suppose that $f \in \B$, and that $S(f) \subset U$, where $U$ is a forward invariant Fatou component of $f$. Then $U$ is completely invariant.
\end{proposition}
\begin{proof}
Let $D \subset U$ be a Jordan domain containing $S(f)$. Since $U$ is forward invariant, $f(U) \subset U$. Moreover, by \cite[Theorem 1]{Herring}, $U \setminus f(U)$ contains at most two points. Hence $f(U)$ meets $D$, and so $U$ meets $f^{-1}(D)$. Since $S(f) \subset D$, it follows by \cite[Proposition 2.9]{BFR} that $f^{-1}(D)$ is connected. Hence $U$ contains $f^{-1}(D)$.
 
Suppose that $V$ is a Fatou component of $f$ such that $f(V) \subset U$. Again, by \cite[Theorem 1]{Herring}, $f(V)$ meets $D$, and so $V$ meets $f^{-1}(D)$. Thus $V = U$ as required.
\end{proof}
We also require some additional notation. If $f$ is a transcendental entire function, then a Fatou component, $U$, is a \emph{wandering domain} if $f^n(U) \cap f^m(U) = \emptyset$, for $n, m \geq 0$. It is easy to see that if $U$ is a wandering domain, then all  limit functions of $\{f^n|_U\}$ are constants in $\hat{\C}$; we denote the set of these constants by $\omega(U)$. It was shown in \cite{limfunctionswd} that $\omega(U) \subset P'(f) \cup \{\infty\}$; here $P'(f)$ denotes the set of finite limit points of $P(f)$.

We are now able to prove Theorem~\ref{theo:comps}. We denote by $I(f)$ the \emph{escaping set} of a transcendental entire function $f$, which is defined by
\[
I(f) := \{ z \in \C : f^n(z) \rightarrow \infty \text{ as } n \rightarrow \infty \}.
\]
\begin{proof}[Proof of Theorem~\ref{theo:comps}]
Since $f \in \B$, we can deduce, by \cite{EandL}, and the classification of periodic Fatou components \cite[Theorem 6]{bergweiler93}, that the only possible periodic Fatou components of $f$ are cycles of attracting basins, parabolic basins, or Siegel discs. Any cycle of attracting or parabolic basins meets $S(f)$, and the boundary of any Siegel disc lies in the postsingular set \cite[Theorem 7]{bergweiler93}. We can deduce that $f$ has only one such cycle, which is not a Siegel disc and contains both $D$ and $U$. 

Since $U$ is forward invariant, we can deduce that $P(f) \subset D \subset \overline{U}$. In addition, $S(f) \subset D \subset U$, and it follows by Proposition~\ref{prop:comps} that $U$ is completely invariant.

If $U$ is a parabolic basin then $P(f) \setminus U = \{\zeta\}$, where $\zeta$ is the parabolic fixed point of $f$. 
Since $f \in \B$, $f$ does not have any wandering domains in $I(f)$ \cite{EandL}.

Suppose, by way of contradiction, that $V$ is a wandering domain of $f$. It follows from the result of \cite{limfunctionswd} mentioned earlier that $\omega(V) \subset \{\zeta, \infty\},$ and thus $\omega(V)= \{\zeta\}$. Choose a point $z \in V$. By assumption $f^n(z) \rightarrow \zeta$ as $n \rightarrow \infty$. Hence, for a sufficiently large value of $n$, $f^n(z)$ lies in an attracting petal at $\zeta$ and so lies in $U$; see \cite[\S10]{Milnor} for more information on the behaviour of iterates near a parabolic fixed point. This is a contradiction. 

If $U$ is an attracting basin then $P(f) \subset U$, and it can be deduced as above that $U$ is the only Fatou components of $f$. In particular, $f$ is of disjoint type. 
\end{proof}
%
%
%
\section{Proof of Theorem \ref{theo:finitelymanytracts}}
\label{sec:disjoint}
In this section we prove Theorem~\ref{theo:finitelymanytracts}. The notation we set up in the proof of this result, which pertains to any class $\B$ function with finitely many tracts, will remain in place throughout the paper.
\begin{proof}[Proof of Theorem~\ref{theo:finitelymanytracts}]
Reducing $D$ slightly, if necessary, we can assume that we have $\overline{D} \subset F(f)$, and so, by Theorem~\ref{theo:comps}, $\overline{D} \subset U$ and $U$ is unbounded.


\begin{figure}
	\includegraphics[width=16cm,height=10cm]{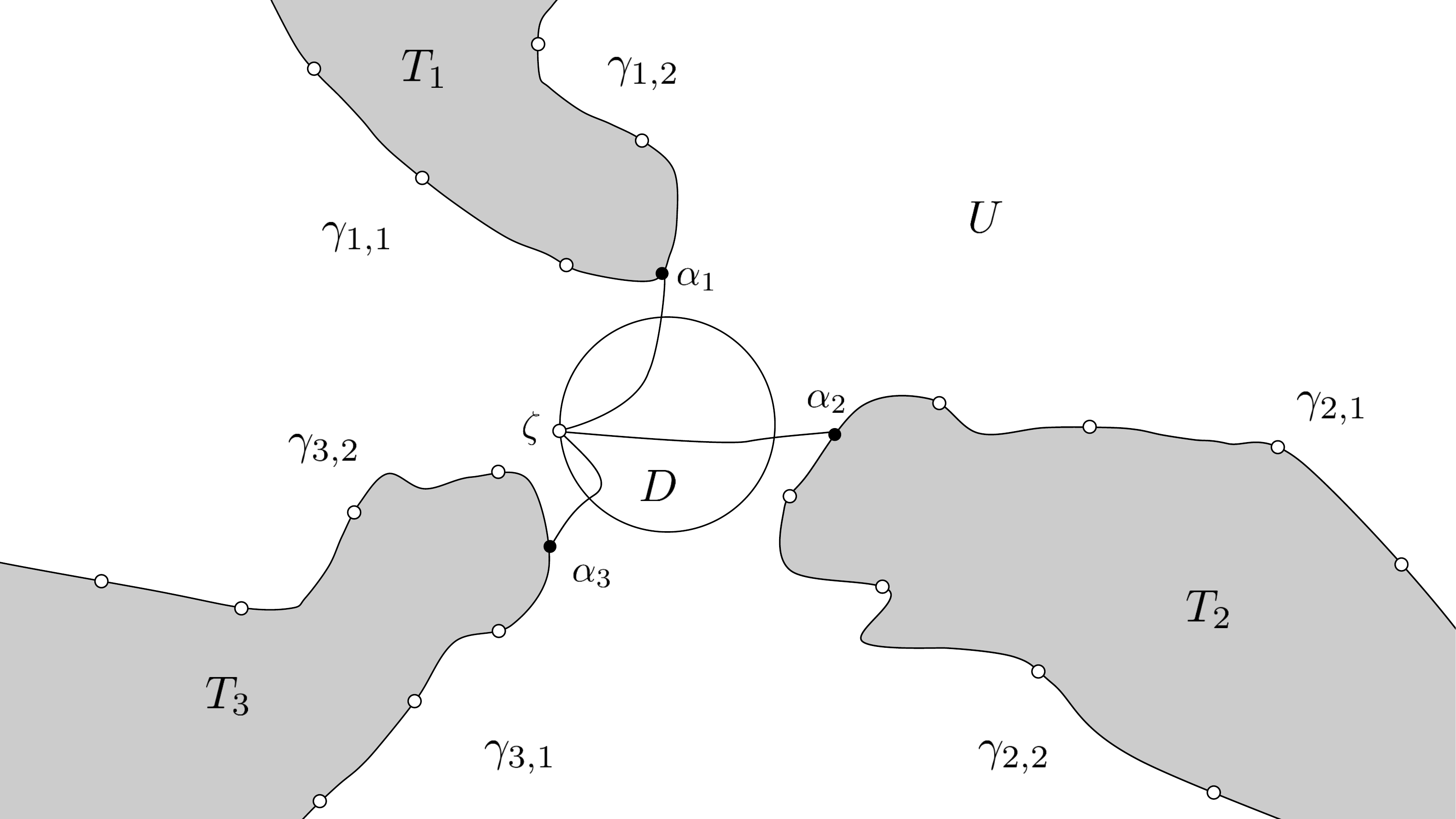}
  \caption{The preimages of $z$ which lie on the half-boundaries $\gamma_{i,k}$, shown as white dots.}\label{fig:graphic0}
\end{figure}

Now note that the boundaries of the tracts are Jordan curves tending to infinity in both directions and contained in $U$. Hence infinity is accessible from $U$. For each $1 \leq i \leq n$, we take a point $\alpha_i \in \partial T_i$. The point $\alpha_i$ divides $\partial T_i$ into two Jordan curves, $\gamma_{i,1}, \gamma_{i,2}\subset \partial T_i$ tending to infinity and sharing the same endpoint, $\alpha_i$, which we call \textit{half-boundaries}. Since the tracts follow a cyclic order at infinity and can be labeled according to that (see \cite[Lemma 2.4(a)]{anna-nuria}) we can assume that they are labeled following an increasing clockwise order and the same for the boundary curves $\gamma_{i,j}$. In other words, the half-boundaries are cyclically ordered as 
\[
\dots, \gamma_{1,1}, \gamma_{1,2}, \dots, \gamma_{i-1,2}, \gamma_{i,1}, \gamma_{i,2}, \gamma_{i+1,1}, \dots, \gamma_{n,1}, \gamma_{n,2}, \gamma_{1,1}, \dots.
\]
Following this order, we will say that two consecutive curves not belonging to the boundary of the same tract are \textit{adjacent}. 

Now take a point $\zeta \in \partial D \setminus \delta$, and we can assume that $f(\alpha_i) \ne \zeta$, for $1 \leq i \leq n$. Then $\zeta \in U$ (since $\overline{D} \subset U$), and all the preimages of $\zeta$ lie on the boundaries of the tracts; we can assume there are infinitely many. More precisely, if $T$ is a tract and $F \subset T$ is a fundamental domain, then we have exactly one preimage of $\zeta$ on $\partial F \cap \partial T$. Now the preimages that belong to $\partial T_i$ belong to either $\gamma_{i,1}$ or $\gamma_{i,2}$. It is easy to see that, in particular, there are infinitely many preimages belonging to $\gamma_{i,1}$ and infinitely many belonging to $\gamma_{i,2}$. 

For each $1 \leq i \leq n$, let $\tau_i\subset U$ be a curve which joins $\zeta$ to $\alpha_i$. For each $1 \leq i \leq n$ and each $k \in \{1, 2\}$, set $\gamma_{i,k}' = \gamma_{i,k} \cup \tau_i$. Then $\gamma_{i,k}' \subset U$ is a curve which joins $\zeta$ to infinity in $U$. Now the set $\bigcup_{i,k} \gamma_{i,k}'$ contains all the preimages of $\zeta$. Also, since $J(f) \subset \overline{\mathcal{T}}$, it follows from Lemma \ref{lem:accesses-tracts} that any two adjacent half-boundaries give rise to the same access to infinity. Hence the curves $\gamma_{i,k}'$ belong to exactly $n$ different accesses to infinity  in $U$. (We are not claiming here that there might not be other accesses to infinity from $U$.)

Next we consider a Riemann map $\phi : \mathbb{D} \to U$ and an inner function $h= \phi^{-1} \circ f \circ \phi$ associated to $f|_U$. By Theorem~\ref{correspondence theorem}(a), we have that each $\phi^{-1}(\gamma_{i,k}')$ lands at exactly one point on $\partial \mathbb{D}$ and we have $n$ such landing points. (This is because any two adjacent half-boundaries belong to the same access to infinity, and so their preimages under $\phi$ land at the same point of $\partial \mathbb{D}$.) Moreover, $\cup_{i,k} \phi^{-1}(\gamma_{i,k}')$ contains all the preimages under $h^{-1}$ of a given point in $\mathbb{D}$, i.e. the point $w= \phi^{-1}(\zeta)$. Since $\zeta$ could be almost any point on $\partial D$, and since we are free to make $D$ slightly smaller, we deduce by Lemma \ref{lem:sing-preimages} that $h$ has exactly $n$ singularities on $\partial \mathbb{D}.$
\end{proof}

%
%
%
\section{The proof of Theorem~\ref{theo:hyperbolic}}
\label{sec:hyp}
\subsection{Preliminaries and sketch of the proof}
First we need to give some definitions, which will be in place in the rest of this paper. Suppose that $f$ is hyperbolic and has finitely many tracts, and that $U$ is a forward invariant Fatou component of $f$. Let $D$ be a large disc such that $P(f) \subset D$ and $D \cap U \ne \emptyset$. Recall that the components of $f^{-1}(\C \setminus \overline{D})$ are the tracts of $f$. 

Choose a point $\zeta \in U \cap \partial D$; there is an open set of points which we can choose here, because we can also make $D$ larger.  Let $U_\zeta$ denote the preimages of $\zeta$ in $U$; in other words, $U_\zeta := f^{-1}(\zeta) \cap U$. Note that, by Lemma~\ref{lem:sing-preimages}, we can assume that $U_\zeta$ is infinite, as otherwise there is nothing to prove. As in the proof of Theorem~\ref{theo:finitelymanytracts}, all the preimages of $\zeta$ lie on the boundaries of the tracts, although we can no longer assume that those boundaries also lie in $U$.

We know that $f$ has only finitely many tracts. For each tract $T$, we choose a point $\alpha \in \partial T$. Then $\partial T \setminus \{\alpha\}$ has two components, each of which is a simple curve to infinity; these are known as half-boundaries. Some (perhaps all) of these curves contain infinitely many points of $U_\zeta$; we call these \emph{good half-boundaries}. Note that since there are only finitely many tracts, at most finitely many points of $U_\zeta$ do not lie on a good half-boundary.

We next outline the strategy of the proof of Theorem~\ref{theo:hyperbolic}. Roughly speaking, in the proof of Theorem \ref{theo:finitelymanytracts} we used portions of the boundaries of the tracts to join all the preimages of $\zeta$ to infinity. These boundaries were in $U$, and the result was then deduced from this. To prove Theorem~\ref{theo:hyperbolic} we adopt a similar approach, and again look to construct curves in $U$ which connect all but finitely many of the points of $U_{\zeta}$ to infinity. In this case we are not able to use the boundaries of the tracts, because these may meet the Julia set. Instead, we show that these curves, which are entirely contained in $U$, can be constructed by ``weaving round'' the components of the Julia set that cross the good half-boundaries. Thus our proof splits into three parts; first we show that there is a suitable bound on the size of certain ``pieces'' of Julia set that cross the good half-boundaries; second we show that this bound implies that the necessary curves can be constructed; and finally we complete the proof using the existence of these curves.


%
\subsection{Bounding ``pieces'' of the Julia set}
In fact, we only need to develop a bound on the size of ``pieces'' of a certain subset of the Julia set. To define these, first let $\mathcal{X}$ denote the collection of all closed, unbounded, connected sets $X$ that have both the following properties;
\begin{itemize}
\item there is an external address $\s$ such that $X \subset J_{\s}$;
\item the iterates of $f$ tend to infinity uniformly on $X$.
\end{itemize}
The following property of the set $\mathcal{X}$ is not stated explicitly in \cite{AnnaLasse}, although it is easily proved from results of that paper.
\begin{proposition}
\label{prop:Xdense}
Suppose that $f \in \B$ and that $P(f)$ is bounded. Then $\mathcal{X}$ is dense in $I(f)$ and $J(f)$.
\end{proposition}
\begin{proof}
First, by \cite[Corollary 4.5]{AnnaLasse}, we know that any point of $I(f)$ has an external address. It is then an immediate consequence of  \cite[Proposition 4.10]{AnnaLasse} that the elements of $\mathcal{X}$ are dense in $I(f)$. The final conclusion follows since $J(f) = \partial I(f)$.
\end{proof}

We prove a uniform bound, in an appropriate metric, on the size of certain connected subsets of the elements of $\mathcal{X}$. Specifying these subsets requires some additional definitions. Suppose that $X \in \mathcal{X}$. We call each \emph{bounded} component of $X \setminus \partial \mathcal{T}$ a \emph{piece}.  Let $K \subset X$ be a piece, and let $\s$ be its address. If $K \subset \tau_1(\s)$, then we say that $K$ is a \emph{bad} piece. Otherwise we say that $K$ is a \emph{good} piece. We will bound the size of good pieces, hence the terminology, and use a different technique to deal with bad pieces.

We next define the metric in which we will bound the size of good pieces, and state an important property of this metric. Set $\Omega := \C \setminus P(f)$; recall that $P(f)$ is compact. We use $\rho_{\Omega}(z)$ to denote the hyperbolic density in $\Omega$, for $z \in \Omega$, and if $S \subset \Omega$, then we denote the hyperbolic diameter in $\Omega$ of $S$ by 
\[
\operatorname{diam}_\Omega S := \sup_{z, w \in S} d_\Omega(z, w),
\]
where $d_\Omega(z, w)$ is the hyperbolic distance in $\Omega$ from $z$ to $w$.
We use the following estimate on the hyperbolic derivative in $V = f^{-1}(\Omega) \subset \Omega$; this is part of \cite[Proposition 3.1]{AnnaLasse}. Here we denote the hyperbolic derivative in $\Omega$ by 
\[
||Df(z)||_{\Omega}:= |f'(z)|\frac{\rho_{\Omega}(f(z))}{\rho_{\Omega}(z)},  \quad\text{for } z \in V.
\]
\begin{proposition}
\label{prop:exp}
Suppose that $f \in \B$, that $P(f)$ is bounded, and that $\Omega$ and $V$ are as defined above. Then for each $\epsilon > 0$ there exists $\Lambda > 1$ such that
\[
||Df(z)||_{\Omega} \geq \Lambda, \quad\text{for } z \in V \text{ with } \operatorname{dist}(z, P(f)) \geq \epsilon.
\]
\end{proposition}

\begin{remark}\normalfont
Note that we could alternatively here have used the estimate for hyperbolic functions given in \cite[Lemma 5.1]{Rigidity}; see also \cite[Proposition 2.2]{BFR}. This leads to an equivalent but slightly different proof.
\end{remark}

We now prove the following.
\begin{proposition}
\label{prop:hyp-julia}
Suppose that $f$ is hyperbolic. Then there is a choice of $D$ such that there exists $\ell > 0$ with the following property. If $K$ is a good piece, then
\begin{equation}
\label{kbound}
\operatorname{diam}_\Omega(K) \leq \ell.
\end{equation}
\end{proposition}

\begin{proof}
A rough outline of the proof is as follows; recall the definition of fundamental tails in Subsection~\ref{subs:tails}. First we show that, with a suitable definition of the set $D$, there is an upper bound for the hyperbolic diameter in $\Omega$ of $(\tau_{n+1}(\s) \setminus\tau_{n}(\s))$, for each external address $\s$ and $n \in \N$; see \eqref{nickbound} below. We then use this to obtain a uniform bound on $\operatorname{diam}_{\Omega} (K)$, where $K$ is a good piece.

We want to be able to apply Proposition~\ref{prop:exp} to the collection of fundamental tails. Suppose that $\s$ is an external address and $n \in \N$. Note that $f(P(f)) \subset P(f)$, and so $\Omega \subset f(\Omega)$. Hence
\[
f^n(\tau_n(\s)) = W_0 \subset \Omega \subset f^{n-1}(\Omega).
\]
It follows that $\tau_n(\s) \subset f^{-1}(\Omega) = V$. So in order to apply Proposition~\ref{prop:exp} it remains to ensure that the fundamental tails are at least some fixed Euclidean distance from $P(f)$.

It is known that there is a bounded open set $G$, which is not necessarily connected, containing $P(f)$, and such that $f(\overline{G}) \subset G$; see, for example, the comments following \cite[Proposition 2.6]{HMB}. Increasing the size of the Jordan domain $D$, if necessary, we can assume that $\overline{G} \subset D$. Hence
\[
f^n(\tau_n(\s)) = W_0 \subset \C \setminus \overline{G} \subset f^n(\C \setminus \overline{G}).
\]
It follows that $\tau_n(\s) \subset \C \setminus \overline{G}$. Thus there exists $\epsilon > 0$ such that all fundamental tails are a distance at least $\epsilon$ from $P(f)$.

It then follows from Proposition~\ref{prop:exp} that there exists $\Lambda>1$ such that, for each external address $\s$, we have
\begin{equation}
\label{hypexpansion}
||Df(z)||_{\Omega} \geq \Lambda, \quad\text{for } z \in \bigcup_{n \in \N} \tau_n(\s).
\end{equation}

Suppose now that $\s = F_0 F_1 \ldots$ is an external address. Suppose that $n \in \N$. It follows from the definition of $\tau_{n+1}(\s)$ that $f^n(\tau_{n+1}(\s) \setminus \tau_{n}(\s)) = F_n \cap D$.

Now, only finitely many fundamental domains intersect $D$. Moreover, we have that $\operatorname{dist}(F, P(f))\geq \varepsilon$ for any fundamental domain $F$. We deduce that there exists $L>0$ such that $\operatorname{diam}_{\Omega} (F \cap D) \leq L$, for any fundamental domain $F$.
Combined with (\ref{hypexpansion}), this implies that
\begin{equation}
\label{nickbound}
\operatorname{diam}_{\Omega}(\tau_{n+1}(\s) \setminus \tau_{n}(\s)) \leq \frac{L}{\Lambda^n}, \quad\text{for } n \in \N.
\end{equation}

This completes the first stage of the proof. Next, let $$\ell := \frac{L}{\Lambda(\Lambda-1)} = L\left(\frac{1}{\Lambda} + \frac{1}{\Lambda^2} + \ldots\right).$$

Suppose that $K \subset X \in \mathcal{X}$ is a good piece, and that all points of $K$ have external address $\s$. We will now use \eqref{nickbound} to prove \eqref{kbound}. Since the iterates of $f$ escape to infinity uniformly on $K$, there exists $N \geq 2$ such that $K \subset \tau_N(\s)$. Since $\tau_1(\s) = F_0$ and $K$ is a good piece, $K \cap \tau_1(\s) = \emptyset$.

We now claim that 
\begin{equation}
\label{aclaim}
K \subset \bigcup_{n\in\N} \tau_{n+1}(\s) \setminus \tau_{n}(\s).
\end{equation}
To prove this, suppose that $z \in K$. Let $1 < n < N$ be the smallest integer such that $z \notin \tau_n(\s)$ and $z \in \tau_{n+1}(\s)$; there must be such an $n$ by the comments above. Hence $z \in \tau_{n+1}(\s) \setminus \tau_n(\s)$, which completes the proof of \eqref{aclaim}.

Equation \eqref{kbound} is then an immediate consequence of \eqref{aclaim}, together with \eqref{nickbound}.
\end{proof}

From now on we will assume that $D$ is chosen suitably for the conclusions of Proposition~\ref{prop:hyp-julia} to hold. We then use the following consequence of Proposition~\ref{prop:hyp-julia}.
\begin{cor}
\label{cor:hyp-julia}
Suppose that $f$ is hyperbolic. Suppose that $(K_k)_{k \in N}$ is a sequence of good pieces, that $z_k \in K_k$, for $k \in \N$, and that $z_k \rightarrow \infty$ as $k \rightarrow \infty$. Then the sets $K_k$ tend uniformly to infinity as $k$ tends to infinity.
\end{cor}
\begin{proof}
Suppose that this were not the case. Taking a subsequence, if necessary, there exist $r>1$ and points $w_k \in K_k$, for $k \in \N$, such that $|w_k| < r$ for each $k$.

After a conjugacy, if necessary, we can assume that $\{0, 1\} \subset P(f)$. We then have, by the Schwarz-Pick lemma and \cite[Theorem 9.13]{Hayman}, that there is a constant $c>0$ such that 
\[
\rho_\Omega(z) \geq \rho_{\C\setminus\{0,1\}}(z) \geq \frac{1}{|z|(\log|z| + 10\pi)} \geq \frac{c}{|z|\log|z|}, \quad\text{for } z \in \Omega, \ |z| \geq r.
\]

It follows that
\[
\operatorname{diam}_\Omega (K_k) \geq \int_{r}^{|z_k|} \frac{c}{t \log t} \ dt = c \log \frac{\log |z_k|}{\log r} \rightarrow\infty \text{ as } k\rightarrow\infty. 
\]
This is in contradiction to Proposition~\ref{prop:hyp-julia}.
\end{proof}
%
%
%
%
%
\subsection{Constructing the curves}
We next show how to use the conclusions of Corollary~\ref{cor:hyp-julia} to construct the curves mentioned earlier. Our construction is formalised in the following lemma.
\begin{lemma}
\label{lemm:mainclaim}
Suppose that $f$ is hyperbolic and has finitely many tracts, that $U$ is a forward invariant Fatou component of $f$, that $U_\zeta$ is as defined above (i.e. the set of preimages of $\zeta \in \partial D$ in $U$), and that 
$\gamma$ is a good half-boundary. Then there is a simple curve $\tilde{\gamma} \subset U$, which joins a finite point to $\infty$, 
and with the property that $(\gamma \cap U_\zeta) \subset \tilde{\gamma}.$
\end{lemma}
\begin{proof}
Let $T$ be a tract, and let $\gamma\subset\partial T$ be a good half-boundary. Since $\gamma$ is a simple curve from a finite point to infinity, which by assumption contains infinitely many points of $U_\zeta$, we can label these points $w_1, w_2, \ldots$ in an obvious order.  For $i \in \N$, we let $\eta_i$ be the portion of $\gamma$ from $w_i$ to $w_{i+1}$, and we also let $\mathcal{K}_i$ denote the collection of all pieces whose closure meets $\eta_i$. It is easy to see that the closure of each element of $\mathcal{K}_i$ is a continuum in $J(f)$ that meets $\eta_i$.

As we discussed at the beginning of this section, our goal is to join $\zeta$ to infinity with a curve in $U$ which is ``close'' to $\gamma$ and contains all the points $w_1, w_2, \ldots$. (In the proof of Theorem~\ref{theo:finitelymanytracts} these curves were just the half-boundaries $\gamma_{i,j}$.) We meet a problem with this goal if there are components of $J(f)$ which meet $\gamma$ and somehow ``obstruct'' our attempts to construct the necessary curve. 

Set
\[
\tau_i = \sup \{ \min |z| : z \in \beta \}, \quad\text{for } i \in \N,
\]
where the supremum is taken over all simple curves $\beta \subset U$ that join $w_i$ to $w_{i+1}$. Roughly speaking, $\tau_i$ measures how small the modulus has to be at some point on a curve in $U$ that joins the two preimages $w_i$ and $w_{i+1}$.

Suppose that $\tau_i\rightarrow\infty$ as $i\rightarrow\infty$. Then we can construct the required curve $\tilde{\gamma}$ by taking a union of curves in $U$ that join all the points of $U_\zeta \cap \gamma$, two at a time, so that $\tilde{\gamma}$ accumulates only at infinity. It follows that we can assume, by way of contradiction, that there is a subsequence $(n_i)_{i\in\N}$ and $L > 0$ such that $\tau_{n_i} \leq L$, for $i \in \N$.

For each $i \in \N$, set $\sigma_i := \min\{|z| : z \in \eta_i\}$. 
We now claim the following.
\begin{claim}
\label{a:claim}
If $i \in \N$ is such that $\sigma_i > \tau_i$, then $\tau_i \geq c_i$, where $c_i := \inf \{ |z| : z \in \mathcal{K}_i \}.$
\end{claim}
Note that, in fact, it can be shown that $\tau_i = c_i$; however we do not need this, and so the proof is omitted.

Before proving Claim~\ref{a:claim}, we show how to complete the proof of Lemma~\ref{lemm:mainclaim}. Since we have assumed that $\tau_{n_i} \leq L$, for $i \in \N$, and since $\sigma_i \to \infty$ as $i \to \infty$, the condition of Claim~\ref{a:claim} is satisfied for infinitely many values of $i$. It follows that for infinitely many $i \in\N$ there is a component $K_i$ of $\mathcal{K}_i$ that contains a point of modulus less than $L+1$. By the comments earlier, the closure of $K_i$ meets $\eta_i$, and so $K_i$ contains a point $z_i$ with $z_i \rightarrow\infty$ as $i \rightarrow\infty$. We show that this gives a contradiction. Suppose first that for infinitely many values of $i$, $K_i$ is a bad piece. By definition, $K_i$ is contained in a fundamental domain $F_i$ that meets $\eta_i$. Since only finitely many fundamental domains can meet a compact set, this is a contradiction. We can assume, therefore, that $K_i$ is a good piece for all sufficiently large values of $i$. This is now an immediate contradiction to Corollary~\ref{cor:hyp-julia}.

It remains, then, to prove Claim~\ref{a:claim}. Suppose, contrary to the claim, that $\tau_i < c_i$. Note that $\tau_i > 0$, since $w_i$ and $w_{i+1}$ both lie in $U$, and $U$ is open. Choose $r \in (0, \tau_i)$. Then there is a simple curve $\beta \subset U$, joining $w_i$ and $w_{i+1}$, such that $\beta$ does not meet $B(0, r)$. Let $V$ denote the union of the bounded components of the complement of $\beta \cup \eta_i$. Clearly $V$ is open, but need not be connected. Each component of $V$ is a Jordan domain, the boundary of which meets $\eta_i$. 

Choose $r' \in (\tau_i, \min\{c_i, \sigma_i\})$. Note that $\beta$ contains a point of modulus at most $\tau_i$, and also a point of modulus greater than $r'$. Hence $\beta$ meets $\partial B(0, r')$ and so $V$ also meets $\partial B(0, r')$. 

\begin{figure}
	\includegraphics[width=16cm,height=10cm]{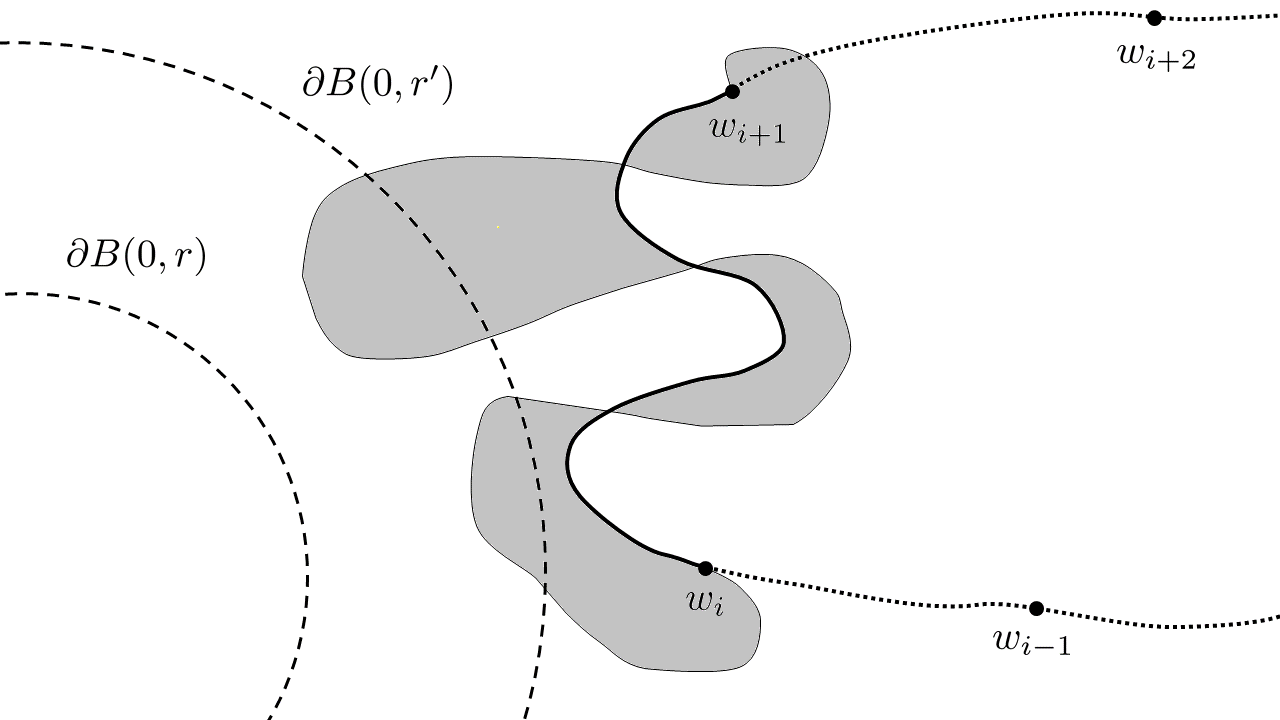}
  \caption{An illustration of the construction in the proof of Lemma~\ref{lemm:mainclaim}. The curve $\eta_i$ is shown black, and the boundary of the rest of the tract is black and dotted. The curve $\beta$ is the thin
	curve from $w_i$ to $w_{i+1}$. The components of $V$ are shown in grey.\label{fig:graphic}}
\end{figure}

Our goal now is to construct a simple curve $\beta' \subset U$, joining $w_i$ and $w_{i+1}$, such that $\beta'$ does not meet $B(0, r')$. This then gives a contradiction, since $r' > \tau_i$.

Each component of $\partial B(0, r') \cap V$ is a cross-cut of $V$, i.e. an open arc lying in $V$ whose closure is an arc with exactly two endpoints in $\partial V$. We claim that each such cross-cut lies in $U$. For, suppose that $T$ is a component of $\partial B(0, r') \cap V$ that meets $\C\setminus U$. Let $V'$ be the component of $V$ containing $T$. Since $\beta \subset U \subset F(f)$, there is a point $\xi \in V' \cap J(f)$ with $|\xi| < c_i$. By Proposition~\ref{prop:Xdense} there exists $X \in \mathcal{X}$ and a point $\xi' \in V' \cap X$ with $|\xi'| < c_i$.

Since $V'$ is bounded, $X$ must meet $\partial V'$. However, $X$ cannot meet $\beta$, because $\beta \subset F(f)$. Hence $X$ meets $\eta_i \subset \partial\mathcal{T}$. We apply Theorem~\ref{nad} with $X' = X \cup \{\infty\}$, with $E = X' \setminus \partial \mathcal{T}$, and with $K$ being the component of $E$ that contains $\xi'$. We deduce that $\overline{K}$ meets $\overline{X' \cap \partial\mathcal{T}}$. It then follows easily that $\xi' \in K \in \mathcal{K}_i$. This is a contradiction since $|\xi'|  < c_i$.

The curve $\beta'$ mentioned earlier can now be constructed by taking a suitable union of components of $\partial B(0, r') \cap V$ with pieces of $\beta$.
\end{proof}

\subsection{Completing the proof}
Finally, we are now able to prove the main result. As mentioned earlier, our approach here is very similar to the proof of Theorem~\ref{theo:finitelymanytracts}, and so we highlight only the important differences.
\begin{proof}[Proof of Theorem~\ref{theo:hyperbolic}]
As in the proof of Theorem~\ref{theo:finitelymanytracts}, fix a point $\zeta \in U$. First suppose that $T_1$ and $T_2$ are tracts (possibly equal), and also that $\gamma_1 \subset \partial T_1$ and $\gamma_2 \subset \partial T_2$ are good half-boundaries that are adjacent (in the sense given in the proof of Theorem~\ref{theo:finitelymanytracts}). Suppose that $\tilde{\gamma}_1$ and $\tilde{\gamma}_2$ are the unbounded simple curves from Lemma~\ref{lemm:mainclaim} (with the obvious notation). For each $j \in \{1, 2\}$ we let $\tilde{\gamma}_j'$ be the simple curve obtained by appending to $\tilde{\gamma}_j$ a curve in $U$ ending at $\zeta$. We claim that $\tilde{\gamma}_1'$ and $\tilde{\gamma}_2'$ are in the same access to infinity from $U$. 

For, consider the components of $W = \hat{\C} \setminus (\tilde{\gamma}_1' \cup \tilde{\gamma}_2')$. Since $U$ is simply connected, and $(\tilde{\gamma}_1' \cup \tilde{\gamma}_2') \subset U$, no bounded component of $W$ can meet $\partial U$. It follows, by Lemma~\ref{lem:accesses-tracts}, that we can assume that $\tilde{\gamma}_1' \cap \tilde{\gamma}_2' = \{\zeta,\infty\}$, and hence that $W$ has exactly two components, each of which is unbounded. Note that, by construction, one component of $W$, say $W'$, has the property that all components of $W' \cap \mathcal{T}$ are bounded.
Suppose that $W' \cap \partial U \neq \emptyset$. Then $W'$ meets $J(f)$. By Proposition~\ref{prop:Xdense}, there exists $X \in \mathcal{X}$ that meets $W'$. We know that $X \subset \tau_N(\s)$ for some $N \in \N$ and an external address $\s$, and so $X \cap \tau_1(\s)$ has an unbounded component. This is a contradiction as $\partial W' \subset U \subset F(f)$, and all components of $W' \cap \mathcal{T}$ are bounded. It then follows, by Lemma~\ref{lem:accesses-tracts}, that $\tilde{\gamma}_1'$ and $\tilde{\gamma}_2'$ are in the same access to infinity for $U$, as claimed.

The remainder of the proof follows as in the proof of Theorem~\ref{theo:finitelymanytracts}, replacing the half-boundaries of that proof (which lie in $U$) with the curves that result from Lemma~\ref{lemm:mainclaim}, and noting that at most finitely elements of $f^{-1}(\zeta) \cap U$ do not lie on a good half-boundary.
\end{proof}
%
%
%
\bibliographystyle{alpha}
\bibliography{ref}
\end{document}